\documentclass[10pt,oneside]{article}
\usepackage{amsfonts,amssymb,latexsym,amsthm,amsmath}
\usepackage[dvips]{graphics}

\newtheorem{theorem}{Theorem}[section]
\newtheorem{corollary}{Corollary}[section]

\newtheorem{definition}{Definition}[section]
\newtheorem{lemma}{Lemma}[section]

\newtheorem{problem}[theorem]{Problem}

\begin{document}

\title{Growth rate of an endomorphism of a group}

\author {Kenneth J. Falconer\\
Mathematical Institute\\
University of St Andrews\\Fife KY169SS Scotland
\and
Benjamin Fine\\
Department of Mathematics\\
Fairfield University\\Fairfield, Connecticut 06430,
United States
\and
Delaram Kahrobaei\\
Doctoral Program in Computer Science\\
CUNY Graduate Center,
City University of New York\\New York,NY 10016,
United States\\
Mathematics Department\\
New York City College of Technology, CUNY}
\date{ }

\maketitle

\begin{abstract}
In [B] Bowen defined the growth rate of an endomorphism of a finitely generated group
and related it to the entropy of a map $f:M \mapsto M$ on a compact manifold. In this note we study the purely group theoretic aspects of the growth rate of an endomorphism of a finitely generated group. We show that it is finite and bounded by the maximum length of the image of a generator. An equivalent formulation is given that ties the growth rate of an endomorphism to an increasing chain of subgroups.  We then consider the relationship between growth rate of an endomorphism on a whole group and the growth rate restricted to a subgroup or considered on a quotient.We use these results to compute the growth rates on direct and semidirect products. We then calculate the growth rate of endomorphisms on several different classes of groups including abelian and nilpotent.
\end{abstract}

\section{Introduction}

In [B] Bowen defined the growth rate of an endomorphism of a finitely generated group
and related it to the entropy of a map $f:M \mapsto M$ on a compact manifold.  In particular
he showed that if $f:M \mapsto M$ is a map of a compact manifold and $h(f)$ is its entropy then
$h(f) \ge \log \mu$ where $\mu$ is the growth rate of $f_\star$ on the fundamental group $\pi_1(M)$.

In this paper we consider the purely group theoretic aspects of the growth rate of an endomorphism of
a finitely generated group.  As pointed out in [MS] very little is known about the metric properties of an
endomorphism of a finitely generated group except in the case of a free group. Our aim in this study is to elaborate on and clarify the ideas in Bowen's paper and to present several new results.  In [MS] several different characteristics of the dynamics of an automorphism of a free group were introduced.

In the next section we introduce Bowen's growth rate of an endomorphism and prove several of the statements mentioned in that note. We  show that the growth rate is finite and bounded by the maximum length of the image of a generator. We then provide an equivalent formulation of growth rate that ties the growth rate of an endomorphism on a group $G$ to an increasing chain of subgroups. In section 3 we consider the relationship between growth rate of an endomorphism on a whole group and the growth rate restricted to a subgroup or considered on a quotient. Using these results we then calculate the growth rate of endomorphisms on several different classes of groups including abelian and nilpotent.  We also correct a mistake in a statement in [B]. These results are also used to compute the growth rates on direct and semidirect products.

\section{Growth Rate of an Endomorphism }

In this section we consider the notion of growth rate of an
endomorphism of a finitely generated group in the sense of Bowen
([B], also [H] page 182). In particular we prove
several of the results discussed but not proved in the above
mentioned note.

\begin{definition} Let $\Gamma$ be a finitely generated group with
generating set $S=\{s_1, s_2, \cdots, s_n\}$; for an element
$\gamma \in \Gamma$ we define the length of the shortest word in
the letters $S \cup S^{-1}$ which represents $\gamma$ denoted by
$|\gamma|_S$. Let $\alpha: \Gamma \rightarrow \Gamma$ be an
endomorphism of $\Gamma$. Now define the growth rate of $\alpha$,
denoted by $GR(\alpha)$, as follows:
$$GR(\alpha) = \sup_{\gamma \in \Gamma} \{ \limsup_{m \rightarrow
\infty} \sqrt[m]{|\alpha^m \gamma|}\}$$
\end{definition}

Our first result shows that the growth rate is finite and that it suffices to consider growth on a generating system.

\begin{theorem} \label{generalised}
Let $\Gamma$ be a finitely generated group generated by $S= \{s_1,
\cdots, s_n\}$ and suppose that $\alpha: \Gamma \rightarrow \Gamma$ is an
endomorphism of $\Gamma$. Then:

$\hphantom{xx}$ (1) $GR(\alpha) = \lim_{m \rightarrow \infty} \sqrt[m]{K_m} =
\inf_{m} \sqrt[m]{K_m}$
where
$$K_m = \max_{1 \leq i \leq n} |\alpha^m s_i|.$$

$\hphantom{xx}$ (2) $GR(\alpha) \le k$ where $k$ is the maximum length of $\alpha(s_i)$ for $1 \le i \le n$.

$\hphantom{xx}$ (3) $GR(\alpha^n) = {GR(\alpha)}^n$ for $ n \geq 1$.
\end{theorem}

\begin{proof} (1) Let $\Gamma = gp(S)$ where $|S|<\infty$, and $\alpha, \beta:
\Gamma \rightarrow \Gamma$ be two endomorphisms of $\Gamma$.
If $K(\alpha) = \max_{(\gamma \in S \cup S^{-1})} |\alpha(\gamma)|$ and $K(\beta)$ is defined analogously then it is straightforward that
$$K(\alpha \circ \beta) \leq K(\alpha) \cdot K(\beta).$$

In particular then it would follow that $K(\alpha^{n+m}) \leq K(\alpha^n) \cdot K(\beta^m)$
and so
$$\lim_{n \rightarrow \infty} {K(\alpha^n)}^{1/n} = \inf
{K(\alpha^n)}^{1/n}$$ exists and equals $K_m$.

(2) Using part (1) we can prove part (2) that the growth rate of an endomorphism $\alpha$ is bounded by the maximum length of $\alpha$ on a generating system.

For each $i = 1,...,n$ let $k_i = |\alpha(s_i)|$ and let $k$ be the maximum value of $k_i$ .  Each $\alpha(s_i)$ can be expressed as a minimal word on $\{s_1,...,s_n\}$ say $W_i(s_1,...,s_n)$ of length $k_i$.  Since $\alpha$ is an endomorphism we have
$$\alpha(W_i(s_1,...,s_n)) = W_i(\alpha(s_1),...,\alpha(s_n))$$
and then suitably reduced.  Consider then
$$ \alpha^2(s_i) = \alpha(W_i(s_1,...,s_n)) = W_i(\alpha(s_1),...,\alpha(s_n)).$$
It follows that $|\alpha^2(s_i)| \le k_i k \le k^2$ and so Inductively then
$$\alpha^m(s_i) \le k^m.$$
Let $K_m = \max_{1\le i \le n} |\alpha^m(s_i)|$.  Then $K_m \le k^m$.  From part (1)
$$ GR(\alpha) = \lim_{m \to \infty} (K_m)^{\frac{1}{m}} \le \lim_{m \to \infty} (k^m)^{\frac{1}{m}} = k$$
establishing (2).

The same type of argument proves part (3). We have first
$\lim_{n \rightarrow \infty} {K(\alpha^n)}^{1/n} = \inf
{K(\alpha^n)}^{1/n}$.
Fix $m$, and let $n=r \cdot m$. Then

$$GR(\alpha) = \lim_{r \rightarrow \infty} {K(\alpha^{r \cdot
m})}^{1/{r \cdot m}} = \lim_{r \rightarrow \infty} [K((\alpha^m)^r)^{1/r}]^{1/m} = GR(\alpha^m)^{1/m}.$$
\end{proof}

We close this section by presenting an equivalent formulation of growth rate of an endomorphism which ties the growth rate of an endomorphism on a group $G$ to an increasing chain of subgroups of $G$.

\begin{definition}\label{falconer}
Let $\Gamma$ be a group and $\alpha: \Gamma \rightarrow \Gamma$ be
an endomorphism of $\Gamma$. Given $r \in \mathbb{R}$, fix
$\alpha$ for every $r >1$. Define
$$H_r = \{\gamma \in \Gamma: \overline{\lim}_{n \rightarrow \infty}
\sqrt[n]{|\alpha^n (\gamma)|} \leq r\}.$$
\end{definition}

The following lemma is straightforward.

\begin{lemma} The growth rate of $\alpha$ is the smallest value of $r$ such that
$H_r = \Gamma$.
\end{lemma}

The sets $H_r$ form an increasing chain of subgroups.

\begin{lemma} $H_r$ is a subgroup of $\Gamma$ (for all
$r>1$).
\end{lemma}

\begin{proof} Suppose that $\gamma_1, \gamma_2 \in H_r$ and let $\epsilon >0$.
Then if $n$ is sufficiently large
$$ |\alpha^n(\gamma_1)|^{1/n} \leq r + \epsilon \text{ and } |\alpha^n(\gamma_2)|^{1/n} \leq r + \epsilon.$$
Then
$$ |\alpha^n(\gamma_1 \gamma_2)| = |\alpha^n(\gamma_1)
\alpha^n(\gamma_2)| \leq |\alpha^n(\gamma_1)|
+|\alpha^n(\gamma_2)| \leq 2(r + \epsilon)^n$$
which implies that
$$ {|\alpha^n(\gamma_1 \gamma_2)|}^{1/n} \leq 2^{1/n} (r+
\epsilon).$$
Since $2^{1/n} \rightarrow 1$ so that $$
\overline{\lim}_{n \rightarrow \infty} {|\alpha^n(\gamma_1
\gamma_2)|}^{1/n} \leq r+ \epsilon \text { for all } \epsilon > 0.$$
Therefore $\gamma_1 \gamma_2 \in H_r$.
If $\gamma \in H_r$,
$|\alpha^n(\gamma^{-1})|=|\alpha^n(\gamma)^{-1}|$ and hence
$\gamma^{-1} \in H_r$.
\end{proof}

\section{Growth Rate on Subgroups and Quotients}

Suppose that $H$ is a finitely generated subgroup of a finitely generated group $G$ and that $\alpha$ is an endormorphism on $G$. If $\alpha(H) \subset H$ then $\alpha$ can be considered as an endomorphism also on $H$.  If this is the case we denote by $GR(\alpha|_H)$ the growth rate of $\alpha$ restricted to $H$.  If further $H$ is normal in $G$ with then $\alpha$ can be considered as an endomorphism on the quotient $G/H$.  We use $GR(\alpha \text{ on } G/H)$ to denote the growth rate of an endomorphism $\alpha$ on $G$ considered as an endomorphism on this quotient.

We now consider the relationship between these various growth rates. First we look at subgroups of finite index.

\begin{theorem} Suppose that $\Gamma$ is a group generated by $S= \{s_1,\cdots, s_n\}$, and $\alpha: \Gamma \rightarrow \Gamma$ is an
endomorphism of $\Gamma$. If $H \leq \Gamma$ is a subgroup of finite index in $\Gamma$ such that $\alpha(H) \subset H$, then
$GR(\alpha|_H) = GR(\alpha)$.
\end{theorem}

\begin{proof}
Let $H$ be a subgroup of finite index in $\Gamma$.  Since $\Gamma$ is finitely generated it follows that $H$ is also finitely generated.  Suppose then that $H= \langle h_1, h_2, \cdots, h_r \rangle$. Suppose that $\alpha: \Gamma \rightarrow \Gamma$ is an endomorphism of
$\Gamma$ such that $\alpha (H) \subset H$
 Clearly,
$$GR(\alpha|_H) \leq GR(\alpha).$$
Since $H$ is a subgroup of index $s$ in $\Gamma$, we have
$$\Gamma = g_1 H \cup \cdots \cup g_s H$$
where we suppose
$$\Gamma = gp(h_1, \cdots, h_r, g_1, \cdots, g_s).$$

Put
$$ K_m = \max_{h_i} |\alpha^m (h_i)| \text{ and }  C= \max_{g_i} \{|w|: \alpha(g_i) = g_k w\}$$
for some $k$ where $w$ is a reduced word on $\{h_1, \cdots, h_r \}$.

Then for any $g_j$ we have
$$ \alpha(g_j) = g_{j_1} w_1 \;(\text{for some } 1 \leq j_1 \leq s, \;\; w_1 \in H, \; |w_1| \leq C)$$
Iterating:
$$ \alpha^2 (g_j) = \alpha(g_{j_1}) \alpha(w_1) = g_{j_2} w_2 \alpha(w_1) \; (1 \leq j_2 \leq s, |w_2| \leq C)$$
$$ \alpha^3 (g_j) = \alpha(g_{j_2}) \alpha(w_2) \alpha^2 (w_1) = g_{j_3} w_3 \alpha(w_2) \alpha^2 (w_1)$$
$$ \vdots$$
$$ \alpha^m(g_j) = g_{j_m} w_m \alpha(w_{m-1}) \alpha^2 (w_{m-2})
\cdots \alpha^{m-1} (w_1)$$
Then
\begin{eqnarray*}
|\alpha^m (g_j)| \leq 1 + C + C K_1 + C K_2 +
\cdots + C K_{m-1}.
\end{eqnarray*}
Since
\begin{eqnarray*}
{K_m}^{1/m} \rightarrow K = GR(\alpha|_H).
\end{eqnarray*}
For all $\epsilon > 0$ there is $a$ such that
\begin{eqnarray*}
K_m \leq a (K + \epsilon)^m
\end{eqnarray*}
So
\begin{eqnarray*}
|\alpha^m(g_i)| &\leq& 1 + C + C a (K+\epsilon)+ C a (K +
\epsilon)^2 + \cdots + C a (K + \epsilon)^{m-1}\\
&\leq& \frac{Ca(K+\epsilon)^m}{1 - \frac{1}{K+\epsilon}}.
\end{eqnarray*}
\end{proof}

Clearly the growth rate on a quotient must be smaller. That is:

\begin{lemma} Let $H$ be a finitely generated normal subgroup of a finitely
generated group $\Gamma$, and $\alpha:\Gamma \rightarrow \Gamma$
be a group endomorphism such that $\alpha(H) \subset H$. Then
$$GR(\alpha \text{ on } {\Gamma}/H) \leq GR(\alpha)$$
\end{lemma}

\begin{proof} Suppose that $\Gamma = gp(g_1, \cdots, g_k)$, so $\Gamma/H = gp(H
g_1, \cdots, H g_k)$. Then
$$\alpha^n(Hg_i)=H\alpha^n(g_i)$$
So $GR(\alpha \text{ on } {\Gamma}/H) \leq GR(\alpha)$.
\end{proof}

The next result shows that if $H$ is a normal subgroup then the growth rate of an endomorphism is bounded by the maximum of the growth rate on $H$ and the growth rate on the quotient $\Gamma/H$.

\begin{theorem} \label{normal}
Let $H$ be a finitely generated normal subgroup of a finitely
generated group $\Gamma$, and $\alpha:\Gamma \rightarrow \Gamma$
be a group endomorphism such that $\alpha(H) \subset H$. Then
$$GR(\alpha) \leq \max \{
GR(\alpha \text{ on } H), GR(\alpha \text{ on } {\Gamma}/H)\}.$$
\end{theorem}

\begin{proof}  Take $\gamma > GR(\alpha \text{ on } H)$ and $\eta >
GR(\alpha \text{ on } \Gamma/H) \; (\gamma, \eta >1)$. Let
\begin{eqnarray*}
\Gamma/H = gp(Hg_1, \cdots, Hg_k).
\end{eqnarray*}
Then there exists $N$ such that if $n \geq N$ then
\begin{eqnarray*}
|\alpha^n(Hg_i)|_{\Gamma/H} \leq \eta^n \text{ for all } i = 1,
\cdots, k.
\end{eqnarray*}
For each $i=1,2, \cdots, k$:
\begin{eqnarray*}
\alpha^N(g_i) \in \alpha^N(Hg_i)=H g_{i1} g_{i2} \cdots g_{ir_i}
\end{eqnarray*}
where $r_i \leq \eta^N$ and $g_{ij} \in \{g_1, \cdots, g_k\}$.
Thus for $i=1, \cdots, k$,
$$\alpha^N(g_i) = h_i g_{i1} g_{i2} \cdots g_{ir_i} \text{ for some }h_i
\in H \hphantom{xx} (1).$$

Fix a set of generators of $H$ which combined with the $g_i$ give
a set of generators of $\Gamma$. From the definition of $GR(\alpha
\text{ on } H)$ there is a constant $C>1$ such that
$$|\alpha^n(h_i)|_H \leq C \gamma^n \text{ for all } i, \text{all }
n=0,1,2, \cdots \hphantom{xx} (2).$$

We now prove by induction that for all $m= 1,2, \cdots$,
$$|\alpha^{mN}(g_i)|_\Gamma \leq C(\gamma^{mN} + \gamma^{(m-1)N}
\eta^N + \cdots + \gamma^N \eta^{(m-1)N}+\eta^{mN}) \hphantom{x} (3)
$$
The statement (3) is true for $m=1$ by (1), (2) above. Assume (3) true for some $m$. From
(1), for each $i$
\begin{eqnarray*}
                    \alpha^{(m+1)N}(g_i) &&= \alpha^{mN}(h_i) \alpha^{mN}(g_{i_1})
\cdots \alpha^{mN}(g_{ir_i})\\
\text{so } |\alpha^{(m+1)N}(g_i)|_\Gamma &&\leq
|\alpha^{mN}(h_i)|_H+\eta^N \max_{1\leq i \leq k}|\alpha^{mN}(g_i)|_\Gamma\\
                                         && \leq C \gamma^{mN} + \eta^N C (\gamma^{mN} + \gamma^{(m-1)N}\eta^N+ \cdots+ \eta^{mN})\\
                                         && \leq C(\gamma^{(m+1)N}+\eta^N \gamma^{mN}+ \cdots+\eta^{(m+1)N}).
\end{eqnarray*}
This completes the induction. It then follows that
\begin{eqnarray*}
&&|\alpha^{mN}(g_i)|_\Gamma \leq C(m+1) \max\{\gamma, \eta\}^{mN},\\
&&\text{so} \lim_{m \rightarrow \infty} |\alpha^{mN}(g_i)|_\Gamma^{1/{mN}} \leq \lim_{m \rightarrow \infty} (C(m+1))^{1/{mN}}
\max\{\gamma, \eta\} = \max\{\gamma, \eta\}.
\end{eqnarray*}
If $h$ is one of the generators of $h$, $|\alpha^n(h)|_\Gamma \leq
|\alpha^n (h)|_H$, so
\begin{eqnarray*}
\lim_{n \rightarrow \infty} {|\alpha^n (h)|_\Gamma}^{1/n} \leq
GR(\alpha \text{ on } H) < \gamma.
\end{eqnarray*}
Hence $GR(\alpha \text{ on } \Gamma) \leq \max\{\gamma, \eta\}$.
This completes the proof.
\end{proof}

Suppose that $\Gamma$ is generated by $s_1,...,s_n$ and $H$ is a subgroup of $\Gamma$.  We say that $H$ has a complement relative to $s_1,..,s_n$ if $H$ has a generating system included in $s_1,...,s_n$.  Then for an endomorphism $\alpha$ on $\Gamma$ we clearly have $GR(\alpha \text{ on } H) \le GR(\alpha)$.  Combining this with Theorem 3.2 we obtain.

\begin{corollary} Let $H$ be a finitely generated normal subgroup of a finitely
generated group $\Gamma$ such that $H$ is not distorted and $\alpha:\Gamma \rightarrow \Gamma$
be a group endomorphism such that $\alpha(H) \subset H$. If $H$ has a complement relative to the generating system of $\Gamma$ then
$$GR(\alpha) = \max \{
GR(\alpha \text{ on } H), GR(\alpha \text{ on } {\Gamma}/H)\}.$$
\end{corollary}

\section{Growth Rate of an Endomorphism for Different Classes of Groups}

In this section we consider the growth rate of an endomorphism for several classes of groups. In particular we look at abelian, nilpotent, and polycyclic groups

\subsection{Abelian groups}

Let $G$ be a free abelian group of rank one, that is an infinite cyclic group, and suppose that $g$ is a generator of $G$. If $\alpha$ is an endomorphism on $G$ then $\alpha(g) = mg$ for some integer $m$ using additive notation.  It follows that $\alpha^n(g) = m^n g$ and therefore   $GR(\alpha) = \lim_{m \to \infty} {|\alpha^m(g)|^{1/m}} = m$.  Now let $\alpha$ be an endomorphism of a free abelian group $\Gamma$ of rank $n$.  The endomorphism $\alpha$ can then be represented by an integral matrix.  If the group $\Gamma$ is tensored with the complex numbers $\mathbb{C}$ and we consider the endomorphism $\alpha \otimes 1$ on $\Gamma \otimes \mathbb{C}$ then the matrix can be diagonalized with its eigenvalues $\lambda_1, \cdots, \lambda_n$ down the diagonal.  The endomorphism is then just multiplication of the generators by the appropriate $\lambda_i$.  The power $\alpha^m$ is then just multiplication by $\lambda_i^m$ and hence the growth rate is $\max(\lambda_i^m)^{\frac{1}{m}} = \max (\lambda_i)$. The same argument applied to a general finitely generated abelian group proves:

\begin{theorem}
For $\Gamma$ a finitely generated abelian group one calculates
\begin{eqnarray*}
GR(\alpha) = \max \{|\lambda_1|, |\lambda_2|, \cdots, |\lambda_r|\}
\end{eqnarray*}
where $\lambda_1, \cdots, \lambda_r$ are the eigenvalues of
$\alpha \otimes 1$ on $\Gamma \otimes \mathbb{C}$.
\end{theorem}

We clarify this with the following explicit example.  Let $G = \mathbb{Z} \times \mathbb{Z}$ the free abelian group of rank 2.  Let $a,b$ be free generators of $G$ and suppose that $\alpha$ is the endomorphism given by
$$\alpha: x \mapsto 2y, y \mapsto x.$$
This endomorphism is then represented by the integral matrix $\left  (\begin{matrix} 0&2\\1&0 \end{matrix} \right )$ which has complex eigenvalues $\pm \sqrt{2}$.  By the theorem then the growth rate is $\sqrt{2}$.  To see this directly notice that if $n$ is even then
$$\alpha^n(x) = 2^{\frac{n}{2}}x, \alpha^n(y) = 2^{\frac{n}{2}}y.$$
Then taking limits we find that the growth rate is $2^{\frac{1}{2}} = \sqrt{2}$.

\subsection{Nilpotent groups}

To compute the growth rate of an endomorphism on a finitely generated nilpotent group we need the following results on the growth rate of an endomorphism on the lower central series.

Recall that the lower central series of a group $\Gamma$ is defined inductively by $\Gamma_1 = \Gamma$ and $\Gamma_j = [\Gamma,\Gamma_{j-1}]$ (see [Ba]).

Now if $g_1, \cdots, g_p \in \Gamma = \Gamma_1$ and $g \in
\Gamma_{j-1}$, where $\Gamma_j =[\Gamma, \Gamma_{j-1}]$ and
$\epsilon = \pm 1$ then we have the following commutator identities (see[Ba]).

\begin{eqnarray*}
[\prod_i {g_i}^{\epsilon_i}, g] = \prod_i {[g_i, g]}^{\epsilon_i}
\;\; (\text{mod } \Gamma_{j+1})
[g, {\prod_i} {g_i}^{\epsilon_i}] = \prod_i {[g_i, g]}^{\epsilon_i} \;\; (\text{ mod } \Gamma_{j+1}).
\end{eqnarray*}

\begin{lemma} (see [Ba])
If $g_1, \cdots, g_p \in \Gamma = \Gamma_1, \; h_1, \cdots, h_q
\in \Gamma_{j-1}$, and if $\epsilon_i = \pm 1, \; \sigma_j = \pm
1$. Then
\begin{eqnarray*}
[\prod_i {g_i}^{\epsilon_i}, \prod_j {h_j}^{\sigma_j}] = \prod_i
[g_i, \prod_j {h_j}^{\sigma_j}]^{\epsilon_i} = \prod_i \prod_j
[g_i, h_j]^{\epsilon_i \sigma_j} \;\; (\text{mod } \Gamma_{j+1})
\end{eqnarray*}
\end{lemma}

The following lemma relating growth rates to the lower central series
 was stated without proof in [B].  It will provide the main calculating tool in determining growth rates in nilpotent groups.

\begin{lemma}
Let $\Gamma$ be a finitely generated group generated by $S= \{s_1,
\cdots, s_n\}$, and $\alpha: \Gamma \rightarrow \Gamma$ be an
endomorphism of $\Gamma$. If $\Gamma_1 = \Gamma$ and $\Gamma_{j+1}
= [\Gamma, \Gamma_j]$, then
\begin{eqnarray*}
GR(\alpha) \geq GR(\alpha \text{ on }
{\Gamma_j}/{\Gamma_{j+1}})^{1/j}.
\end{eqnarray*}
\end{lemma}

\begin{proof}  Let $S_{\Gamma_1}$ be a finite set of generators for
$\Gamma_1 = \Gamma$. Inductively we define
$$S_{\Gamma_j} = \{[g, h] : g \in \Gamma_1 , h \in \Gamma_{j-1}
(\text{for } j = 2, 3, \cdots)\}.
$$
Then $S_{\Gamma_j}$ is finite and  $\{ h\Gamma_{j+1} : h \in
S_{\Gamma_j} \}$ is a set of generators for ${\Gamma_j} / {\Gamma_{j+1}}$. Write
\begin{eqnarray*}
K_j (n) = \max \{|\alpha^n (h) ; \text{ mod } \Gamma_{j+1}|: h \in \Gamma_j \}.
\end{eqnarray*}
Let $g \in S_{\Gamma_1}$ and $h \in S_{\Gamma_{j-1}}$, so $[g, h] \in S_{\Gamma_j}$. Then
\begin{eqnarray*}
\alpha^n ([g,h]) &=& [\alpha^n (g), \alpha^n (h)]\\
                 &=& [g_1, g_2, \cdots g_p, h_1 h_2 \cdots h_q \; (\text{ mod } \Gamma_j)]\\
                 &&(\text{where } g_i \in \Gamma_1, h_j \in \Gamma_{j-1}, p \leq K_1(n), q \leq K_{j-1}(n))\\
                 &=& \prod_{i=1} ^p \prod_{j=1} ^q [g_i, h_j] \; (\text{ mod } \Gamma_{j+1}) (\text{ since } [g_i,\Gamma_j] \subset \Gamma_{j+1}).
\end{eqnarray*}
Hence
\begin{eqnarray*}
K_j(n) \leq 4 p q \leq 4 K_1 (n) K_{j-1} (n)
\end{eqnarray*}
Thus
\begin{eqnarray*}
{K_j(n)}^{1/n} \leq 4^{1/n} K_1(n)^{1/n} K_{j-1}(n)^{1/n}
\end{eqnarray*}
Then
\begin{eqnarray*}
GR(\alpha \text{ on } {\Gamma_j} /{\Gamma_{j+1}}) \leq
GR(\alpha) GR(\alpha \text{ on } {\Gamma_{j-1}}/{\Gamma_j})\\
\leq GR(\alpha)^2GR(\alpha \text { on } {\Gamma_{j-2}}/{\Gamma_{j-1}})\\
\leq ... \leq GR(\alpha)^{j-1}GR(\alpha \text { on } {\Gamma_{1}}/{\Gamma_{2}}))\\
\leq GR(\alpha)^j
\end{eqnarray*}
\end{proof}

We now can compute the growth rate of an endomorphism on a nilpotent group of class $t$.

\begin{theorem}
Suppose $\Gamma$ is a finitely generated nilpotent group of class
$t$. Then

$\hphantom{xx}$(1) $GR(\alpha)  = \max_{1 \leq k \leq t} \{GR(\alpha \text{ on }
    ({\Gamma_k}/{\Gamma_{k+1}}))^{\frac{1}{k}}\}$.

$\hphantom{xx}$(2) $GR(\alpha) = \max \{GR(\alpha \text{ on } \Gamma/{\Gamma_t}), GR(\alpha \text{ on } \Gamma_t)^{1/t} \}$, where $\Gamma/{\Gamma_t}$ is a
finitely generated nilpotent group of class $t-1$ and $\Gamma_t$
is central in $\Gamma$.
\end{theorem}

We note that in Bowen's paper this result was stated incorrectly with the power $\frac{1}{k}$ omitted. There is a complete proof in [B]. The proof is correct but in the final step the possible power is suddenly lost. Bowen's proof proceeds by induction on the nilpotency class $t$ and then uses Lemma 4.4. The power is absolutely necessary as we will show in the following example with the Heisenberg group.

The Heisenberg group is the subgroup of $SL(3,\mathbb{Z})$ given by
\begin{eqnarray*}
\Gamma_1 = \Gamma &&= \{\left(%
\begin{array}{ccc}
  1 & a & b \\
  0 & 1 & c \\
  0 & 0 & 1 \\
\end{array}\right )\}
\end{eqnarray*}
We denote an element of the Heisenberg group as above by $(a,b,c)$. Multiplication is then given by $(a,b,c)(p,q,r) = (a+p,b+q+ar,c+r)$.

The Heisenberg group is nilpotent of class 2 and the lower central series and its quotients are given by

$$\Gamma_1=\{(a,b,c)\},\Gamma_2=\{(0,a,0)\},\Gamma_3=\{(0,0,0)\}.$$
$${\Gamma_1}/{\Gamma_2} = \{(a,b,c)\cdot(0,q,0): q \in \mathbb{Z}\}
= \{(a, \mathbb{Z}, c)\}, {\Gamma_2}/{\Gamma_3}= \Gamma_2.$$

Given real numbers $\lambda, \gamma$, let $\phi: \Gamma \rightarrow \Gamma$ be the map given by
$$\phi(a,b,c) = (\lambda a, \lambda \gamma b, \gamma c).$$
It is straightforward to show that $\phi$ is an endomorphism.

Let $\lambda= \gamma =2$, so that $\phi(a,b,c)\ = (2a, 4b, 2c)$. Using $(1,0,0),(0,1,0),(0,0,1)$ as a set of generators for $\Gamma$ it follows from a direct computation that the growth rate $\phi$ is $\le 2$. However $\phi(\Gamma_2) \subset \Gamma_2$ and again by direct computation
$$GR(\phi \text{ on } \Gamma_2) = GR(\phi \text{ on } \Gamma_2/\Gamma_3) = 4.$$
Hence
\begin{eqnarray*}
\max\{GR(\phi \text{ on } {\Gamma_1}/{\Gamma_2}), GR(\phi \text{
on } {\Gamma_2}/{\Gamma_3})\}  \ne GR(\phi \text{ on } \Gamma).
\end{eqnarray*}

Therefore equality would not hold without the exponent. Since $4^{\frac{1}{2}} = 2$ it does hold with the exponent.
\smallskip



\begin{theorem} $GR(\alpha \text{ on } \Gamma) \leq
\max_{1 \leq k \leq t} GR(\alpha \text{ on }
{\Gamma_k}/{\Gamma_{k+1}})$.
\end{theorem}
\begin{proof}
\begin{eqnarray*}
GR(\alpha \text{ on } \Gamma_t) &\leq& \max\{ GR(\alpha \text{ on
} \Gamma_{t+1}), GR(\alpha \text{ on }
{\Gamma_t}/{\Gamma_{t+1}})\}\\
GR(\alpha \text{ on } \Gamma_{t-1}) &\leq& \max\{GR(\alpha \text{
on } {\Gamma_t}/{\Gamma_{t+1}}), GR(\alpha \text{ on }
{\Gamma_{t-1}}/{\Gamma_{t}})\}\\
GR(\alpha \text{ on } \Gamma_{t-2}) &\leq& \max\{GR(\alpha \text{
on } \Gamma_{t-1}), GR(\alpha \text{ on }
{\Gamma_{t-2}}/{\Gamma_{t-1}})\}\\
GR(\alpha \text{ on } \Gamma_{t-2}) &\leq& \max\{GR(\alpha \text{
on } \Gamma_{t}), GR(\alpha \text{ on }
{\Gamma_{t-1}}/{\Gamma_t}), GR(\alpha \text{ on }
{\Gamma_{t-2}}/{\Gamma_{t-1}})\}\\
&\vdots& \\
GR(\alpha \text{ on } \Gamma_{1}) &\leq& \max\{GR(\alpha \text{ on
} {\Gamma_{1}}/{\Gamma_{2}}), GR(\alpha \text{ on }
{\Gamma_{2}}/{\Gamma_3}), \cdots, GR(\alpha \text{ on }
{\Gamma_{t-1}}/{\Gamma_{t}}),\\
&&GR(\alpha \text{ on } \Gamma_{t})\}.
\end{eqnarray*}
\end{proof}

\section{Group Products}

For direct and free products there is an immediate relationship between the growth of an endomorphism on the whole group and on the factors.

\subsection{Direct Products and Amalgams}

\begin{lemma}
Let $\Gamma =A \times B$ be a direct product of two finitely generated groups $A$ and $B$, Suppose that $\alpha$ is an endomorphism from $\Gamma$ to $\Gamma$   such that either $\alpha(A) \subset A$ or $\alpha(B) \subset B$. Then
\begin{eqnarray*}
GR(\alpha \text{ on } \Gamma) = max \{(GR(\alpha \text{ on } A), GR(\alpha \text{ on } B)\}.
\end{eqnarray*}
\end{lemma}

\begin{proof}
Using Theorem \ref{normal} we have:
\begin{eqnarray*}
GR(\alpha \text{ on } A) \leq GR(\alpha \text{ on } \Gamma) \leq max \{ GR(\alpha \text{ on } B), GR(\alpha \text{ on } A)\}\\
GR(\alpha \text{ on } B) \leq GR(\alpha \text{ on } \Gamma) \leq max \{ GR(\alpha \text{ on } B), GR(\alpha \text{ on } B)\}.
\end{eqnarray*}
\end{proof}

Trivially the same observation works for free products provided the endomorphism maps the factors to themselves.

\begin{lemma}
Let $\Gamma = A \star B$ be a free product of two finite generated groups $A$ and $B$. Suppose that $\alpha$ is an endomorphism from $\Gamma$ to $\Gamma$   such that $\alpha(A) \subset A$ and $\alpha(B) \subset B$. Then
\begin{eqnarray*}
GR(\alpha \text{ on } \Gamma) = \max \{(GR(\alpha \text{ on } A), GR(\alpha \text{ on } B)\}.
\end{eqnarray*}
\end{lemma}

\subsection{Semidirect Products}

The behavior of endomorphisms on semidirect products is more complicated than on direct products, Here we first recall the definition of semidirect product of two groups and
the normal form for elements in such products.

\begin{definition} If $H$ and $Q$ are groups then their {\bf semidirect product} is the group
\begin{eqnarray*}
&& \Gamma = H \rtimes_{\phi} Q \;\; (\phi: Q \rightarrow \text{Aut} H)\\
&& (h, q) (h', q') = (h \phi(q) h', q q')
\end{eqnarray*}
where $\phi$ is a homomorphism from $Q$ to $H$.
\end{definition}

Inductively we calculate the product of $n$ elements of $\Gamma$
as follows:
\begin{eqnarray*}
&& (h_1, q_1)(h_2, q_2) \cdots (h_n, q_n) \\
&& = (h_1 \phi(q_1) h_2 \phi(q_1 q_2) h_3, q_1 q_2 q_3)
\cdots (h_n, q_n)\\
&& = (h_1 \phi(q_1) h_2 \phi(q_1 q_2) h_3 \phi(q_1 q_2 q_3) h_4,
q_1 q_2 q_3 q_4) \cdots (h_n, q_n)\\
&& = (h_1 \phi(q_1) h_2 \phi(q_1 q_2) h_3 \phi(q_1 q_2 q_3) h_4
\cdots \phi(q_1 q_2 \cdots q_{n-1}) h_n, q_1 q_2 \cdots q_n)\\
&& = (h_1 \phi(q_1) h_2 \phi(q_1)\phi(q_2) h_3 \phi(q_1)\phi(q_2)
\phi(q_3) h_4 \cdots \phi(q_1) \phi(q_2) \cdots \phi(q_{n-1}) h_n,
q_1 q_2 \cdots q_n)
\end{eqnarray*}

If both factors, $A$ and $B$, in a semidirect product are abelian groups it is straightforward to compute the growth rate of an endomorphism.

\begin{theorem}
Let $\alpha$ be and endomorphism from the finitely generated group $\Gamma$ to $\Gamma$ where $\Gamma = A \rtimes B$ where $A$ and $B$ are finitely generated abelian groups with $\alpha(A) \subset A$.
Then $GR(\alpha \text{ on } \Gamma)= max \{|\lambda_1|, \cdots, |\lambda_r|, |\mu_1|, \cdots, |\mu_t|\}$, where $\lambda_1, \cdots, \lambda_r$ are the eigenvalues of $\alpha \otimes 1$ on $A \otimes \mathbb{C}$ and $\mu_1, \cdots, \mu_t$ are the eigenvalues of $\alpha \otimes 1$ on $B \otimes \mathbb{C}$.
\end{theorem}
\begin{proof}
\begin{eqnarray*}
&& GR(\alpha \text{ on } B) \leq GR(\alpha \text{ on } \Gamma) \leq \text{ max }\{GR(\alpha \text{ on } A), GR(\alpha \text{ on } B)\}\\
&& max \{|\mu_1|, \cdots, |\mu_t|\} \leq GR(\alpha \text{ on } \Gamma) \leq \text { max }\{|\lambda_1|, \cdots, |\lambda_r|, |\mu_1|, \cdots, |\mu_t|\}
\end{eqnarray*}
\end{proof}

If $\Gamma$ is a cyclic group then the growth rate of an endomorphism on $\Gamma$ is clearly an integer, the eigenvalue for the generator. Using the arguments above, this can be extended to endomorphisms that preserve a polycyclic series. Recall that a
group is called polycyclic if there exists a polycyclic series
for the group, that is a subnormal series of finite length
with cyclic factors.  That is a subnormal series
\begin{eqnarray*}
     1 = P_h \triangleleft P_{h-1} \triangleleft \cdots \triangleleft P_1 =
     \Gamma
\end{eqnarray*}
such that $ {P_{i-1}}/{P_i}$ is cyclic. The integer $h$ is the length of the polycyclic series.   We assume that we have a polycyclic set of generators for $\Gamma$ so that the generators of each $P_i$ are included among the generators of $\Gamma$.
If an endomorphism preserves each $P_i$ we say that it preserves the polycyclic series.

\begin{lemma} Let $\Gamma$ be a finitely generated polycyclic group.  If an endomorphism $\alpha$ on $\Gamma$ preserves a polycyclic series then the growth rate of $\alpha$ is an integer.
\end{lemma}

\begin{proof} The simplest situation is when $\Gamma$ is a cyclic-by-cyclic group,  that is $\Gamma$ has a cyclic normal subgroup $A$ with cyclic quotient $\Gamma /A$.  Let $\alpha$ be an endomorphism of $\Gamma$ which preserves this polycyclic series so that $\alpha(A) \subset A$. By assumption the generator of $A$ is part of the generating system of $\Gamma$ and hence
$$GR(\alpha \text{ on } A) \le GR(\alpha) \le \max\{GR(\alpha \text{ on } A), GR(\alpha \text{ on } \Gamma/A)\}$$
We always have
$$GR(\alpha \text{ on } \Gamma/A) \le GR(\alpha)$$
Therefore here
$$GR(\alpha) = \max\{GR(\alpha \text{ on } A), GR(\alpha \text{ on } \Gamma /A)\}.$$  Since both groups are cyclic, it follows that both growth rates are integers, and hence $GR(\alpha)$ is an integer.  The general result then follows by induction on the length of the polycyclic series.
\end{proof}

\subsection{Semidirect Products and the Distortion Function}

The growth rate of an endomorphism is one of many similar growth functions for a finitely generated group. For a finitely generated group $\Gamma$, Gromov has defined the \textbf{distortion function} $\rho(n)$ of a finitely generated subgroup $H \subset \Gamma$ as follows.  For $n \in \mathbb{N}$ the value
 $\rho(n)$ is the radius of the set of vertices in the Cayley graph of $H$ that are a
distance at most $n$ from the identity in $\Gamma$ (See [Br 1] for more information).

For example if $\phi \in GL(r, \mathbb{Z})$ has an eigenvalue of absolute
value greater than $1$ then $\mathbb{Z}^r$ is exponentially
distorted in $\mathbb{Z}^r \rtimes_\phi \mathbb{Z}$.

In the case of a semidirect product we can relate  $\rho(n)$ to  a type of growth rate, which equals the maximal growth rate of the $\phi(q)$ for all $q$ if $Q$ is abelian.

\begin{lemma}
Let $\Gamma = H
\rtimes_{\phi} Q \;\; (\phi: Q \rightarrow \text{Aut} H)$ be a semidirect product.
Let $\Gamma_H $ and $\Gamma_Q$ be generating sets for  $H$ and $Q$. Write
$$K_m = \max_{(q_{i_k} \in \Gamma_Q \cup \Gamma^{-1} _Q,
h \in \Gamma_H \cup \Gamma^{-1} _H)}
\{|\phi(q_{i_1})\phi(q_{i_2}) \cdots \phi(q_{i_m}) h|_H\}.$$
Then the limit
$$K = \lim_{m \rightarrow \infty} {K_m}^{1/m}$$
exists and is independent of the generating sets.
Moreover, if $Q$ is abelian then
$$K = \max_{q \in Q} GR (\phi(q)).$$
Writing  $\rho(n)$ for the distortion function of $H \subset \Gamma$ we have
$$\lim_{n \to \infty}\rho(n)^{1/n}= K^{1/2}.$$
 \end{lemma}

\begin{proof}
Note that $K_{mn} \leq K_n^m $ so that $K_{mn}^{1/mn}\leq K_n^{1/n} $ from which the limit exists. The limit is independent of the generating set by expressing one set of generators in terms of another in the usual way. If $Q$ is abelian then the expression for $K_m$ may be broken down into powers of the $\phi(q_i)$.

Let $h_i \in \Gamma_H \cup \Gamma^{-1} _{H}$ and $q_i \in \Gamma_Q
\cup \Gamma^{-1} _Q$.
If $(h_1,q_1) \cdots (h_n, q_n) \in H \rtimes (0)$ we have, using the definition of the semidirect product and that $q_1q_2\ldots q_n$ is the identity in $Q$ (writing $[\quad]$ for the ceiling function),
\begin{eqnarray*}
|(h_1,q_1) \cdots (h_n, q_n)|_H &\leq& |h_1|_H+|\phi(q_1)h_2|_H
+|\phi(q_1) \phi(q_2) h_3|_H + \cdots  +|\phi(q_1) \cdots
\phi(q_{n-1})h_n|_H\\
&=& |h_1|_H+|\phi(q_1)h_2|_H
+|\phi(q_1) \phi(q_2) h_3|_H \\
&& +\cdots + |\phi(q_{n}^{-1})\phi(q_{n-1}^{-1})h_{n-1}|_H +|\phi(q_{n}^{-1})h_n|_H\\
&\leq& K_0 + K_1 + K_2 + \cdots + K_{[n/2]}+ \cdots + K_2 + K_1\\
&\leq& C'[1+(K+\epsilon)+(K+\epsilon)^2+ \cdots +
(K+\epsilon)^{[n/2]}\\
&\leq& C (K + \epsilon)^{[n/2]}
\end{eqnarray*}
Therefore $d_H(h) \leq C(K + \epsilon)^{[n/2]}$ where $ n=
d_\Gamma (h)$, so $\rho(n) \leq C(K + \epsilon)^{[n/2]}$, giving that
$\limsup_{n \to \infty}\rho(n)^{1/n}\le K^{1/2}$.

For the opposite estimate,
letting
\begin{eqnarray*}
k &\equiv& (0,q_1)(0,q_2) \cdots (0,q_r)(h,0)(0,{q_r}^{-1}) \cdots
      (0,{q_1}^{-1})\\
  &=& (\phi(q_1) \phi(q_2) \cdots \phi(q_r) h, q_1 q_2 \cdots q_r
  {q_r}^{-1} \cdots {q_1}^{-1}) \in H \rtimes (0),
\end{eqnarray*}
we get $d_\Gamma (k) = 2r +1$ and $d_H(k) =|\phi(q_1) \cdots
\phi(q_r)h|$. We can find $q_1, \cdots, q_r, h$ such that
$d_H(k) = K_r $, from which it follows that
$K^{1/2} \leq \liminf_{n \to \infty}\rho(n)^{1/n}.$
\end{proof}

\section{Future work and Open problems}
\begin{problem}
Let $\Gamma = \{A \star B; C\}$ be a free product of two finitely generated groups $A$ and $B$ amalgamating a subgroup $C$. Suppose that $\alpha$ is an endomorphism from $\Gamma$ to $\Gamma$   such that $\alpha(A) \subset A$ and $\alpha(B) \subset B$. How can one compute the $GR(\alpha \text{ on } \Gamma)$ using the $GR(\alpha \text{ on } A)$, $GR(\alpha \text{ on } B)$, and $GR(\alpha \text{ on } C)$?
\end{problem}
\begin{problem}
Following definition \ref{falconer}, as $r$ ranges over $1 < r < \infty$ for which classes of groups do
we get just finitely many different $H_r$?
\end{problem}
\begin{problem}
Can one compute the growth rate of endomorphisms in the other classes of groups?
\end{problem}

\noindent  \textbf{Acknowledgement.} Delaram Kahrobaei is grateful
to Goulnara Arzhantseva for   helpful discussions in the early stage
of the project. Further, Professor Arzhantseva originally proposed
to look at the problem of the growth rate of  endomorphisms of
various groups, particularly  of semidirect products of  groups.
Kahrobaei also acknowledges the support of Professor Arzhantseva
through a grant by National Swiss Foundation, which made it possible
for her to visit Arzhantseva and de la Harpe at the University of
Geneva. Delaram Kahrobaei was also supported by a grant from
research foundation of CUNY (PSC-CUNY) and City Tech Foundation.

\section{References}

\noindent [Ba] G. Baumslag, {\it Lecture Notes on Nilpotent Groups},
American Mathematical Society, 1971.

\noindent [B] R. Bowen, {\it Entropy and the fundamental group},in:
The structure of attractors in dynamical systems (Proc. Conf. North
Dakota State Univ., Frago, N.D. 1977), Lecture  notes in Math. {\bf
668} (1978),  21--29.

\noindent [Br 1] M.Bridson, {\it On the growth of groups and
automorphisms},  Internat. J. Algebra Comput.  {\bf 15} (2005),
869--874.

\noindent \lbrack H\rbrack  P. de la Harpe, {\it Topics in Geometric
Group Theory},  Chicago Lecture Notes in Mathematics, 2000.

\noindent \lbrack MS\rbrack  A. Myasnikov and V. Shpilrain, {\it
Some metric  properties of automorphisms of groups}, J. Algebra {\bf
304} (2006), 782--792.

\end{document}